\documentclass[11pt]{amsart}

\usepackage{amsthm,amssymb,amsfonts,amsmath,enumerate,graphicx,tikz}
\usepackage[margin=1in]{geometry}

\newcommand{\Z}{\mathbb{Z}}

\renewcommand{\phi}{\varphi}

\newcommand{\hang}{\,^a}

\newcommand\commentout[1]{}

\newtheorem{theorem}{Theorem}[section]

\newtheorem{proposition}[theorem]{Proposition}

\theoremstyle{remark}

\theoremstyle{definition}
\newtheorem{definition}[theorem]{Definition}

\begin{document}

\title{Complete intersection P-partition rings}

\author{Brian Davis}
\address{Department of Mathematics\\
         University of Kentucky\\
         Lexington, KY 40506--0027}
\email{brian.davis@uky.edu}

\date{\today}

%\thanks{THANKS AND ACKNOWLEDGEMENTS HERE}

%%%%%%%%%%%%%%%%%%%%%%%%%%%%%%%%%%%%%%%%%%%%%%%%%%%%%%%%%%%%%%

\maketitle

\begin{abstract}We present an alternate proof of a result of F\'eray and Reiner characterizing posets whose $P$-partition rings are complete intersections. This shortened proof relates the complete intersection property to a simple structural property of a graph associated to $P$. 
\end{abstract}

\section{Introduction}
For a finite partially ordered set $P$, F\'{e}ray and Reiner \cite{PpartitionsRev} present a method for computing the Hilbert series of the ring of $P$-partitions, and consequently its number of linear extensions. Special attention is given to a class of posets that they call forests with duplication. We present a shortened proof of a result in their paper: that the forests with duplication are precisely those posets whose $P$-partition rings are complete intersections. We do this by showing (Proposition \ref{forestTwigs}) that a graph $\Gamma(P)$ associated to $P$ is a disjoint union of edges and isolated vertices if and only if $P$ is a forest with duplications. We then show (Theorem \ref{CIProposition}) that such graphs characterize the $P$-partition rings that are complete intersections.
\begin{definition} A (weak) $P$-partition is a map $f:P\rightarrow \Z_{\geq0}$ with the property that $x\preceq_P y$ implies that~${f(x)\geq f(y)}$. 
\end{definition}
It is easy to see that the sum of two $P$-partitions is again a $P$-partition, so that $P$-partitions form a semigroup. The collection of connected order ideals correspond to a minimal generating set for this semigroup, as described below.

If the cardinality of $P$ is $n$, then to each lower order ideal $J$ of $P$ we associate a $\{0,1\}$-vector in $\Z_{\geq0}^n$, given by \[\chi_{J}=\sum_{i\in J}e_i,\] where $e_i$ is the $i$'th standard basis vector. Note that each $\chi_{J}$ corresponds to a $P$-partition by $f(i)=(\chi_{J})_i$.  

\begin{proposition}\label{generatorsProp}
Each $P$-partition $f$ can be considered as a vector in $\Z_{\geq0}^n$, by \[f=\sum_{i=1}^nf(i)e_i,\] and can be written as a sum of vectors $\chi_{J}$.
\end{proposition}
\begin{proof}
Observe that the subset of $P$ consisting of elements $x$ such that $f(x)\neq0$ is an ideal $J$ of $P$, and that $f-\chi_J$ is again a $P$-partition. Iterating gives a decomposition of $f$ into indicator vectors $\chi_{J}$ of ideals of $P$.
 \end{proof}

F\'{e}ray and Reiner call the associated semigroup ring the ring $R_P$ of $P$-partitions. Let $K$ be a field and $S=K[U_{J_1},\dots,U_{J_n}]$, where $U_{J_i}$ is a variable associated to the connected order ideal $J_i$ of $P$. Under the multiplicative map $\deg(\cdot)$ that sends the variable $U_{J_i}$ to $\chi_{J_i}$, a monomial of $S$ corresponds (non-uniquely) to a $P$-partition. Proposition \ref{generatorsProp} shows that the degree map is surjective, and so we can describe the ring of $P$-partitions as a quotient of $S$ by the binomial ideal $I_P$ generated by differences of monomials in fibers of $\deg(\cdot)$, i.e, $R_P$ has a presentation 
\[0\rightarrow I_P\longrightarrow S \longrightarrow R_P\rightarrow 0.\]

We can give explicit generators for $I_P$. For each pair $J_1$ and~$J_2$ of connected order ideals of $P$, let \[\left\{J_1\cap J_2^{(i)}\;:\;i=1,\dots,k\right\}\] be the set of connected components of the Hasse diagram of $J_1\cap J_2$. If $J_1$ and $J_2$ have non-trivial intersection, i.e., neither ideal contains the other and the intersection is non-empty, then 
\begin{equation}\label{generators}U_{J_1}U_{J_2}-U_{J_1\cup J_2}\cdot U_{J_1\cap J_2^{(1)}}\cdots U_{J_1\cap J_2^{(k)}}\end{equation} is a generator of $I_P$, and the set of all such terms generates $I_P$.

Recall that a ring $K[U_1,\dots,U_n]/I$ is a complete intersection when, for $I=(f_1,\dots,f_s)$, we have that $\overline{f_{i+1}}$ is not a zero divisor in $K[U_1,\dots,U_n]/(f_1,\dots,f_{i})$ for all $i$. A useful characterization of complete intersections is that the Krull dimension of  $K[U_1,\dots,U_n]/(f_1,\dots,f_s)$ is $n-s$.
\section{Forests with duplication}
F\'{e}ray and Reiner introduced forests with duplication and showed that they are precisely the posets whose $P$-partition rings are complete intersections. 
\begin{definition}
A forest with duplications is a poset formed from one element posets under the following three operations:
\begin{enumerate}[$(i)$]
\item Disjoint Union: Given posets $P$ and $Q$, union the sets and relations to form $P\oplus Q$.
\item Hanging: Given posets $P$ and $Q$, select an element $a$ in $P$ and introduce relations $q\prec a$ for each element $q$ of $Q$, then close transitively to form $P\hang Q$.
\item Duplication: Given a poset $P=Q_1\hang \,Q_2$ with $a$ minimal (not necessarily $\hat{0}$) in $Q_1$, introduce the new element $a'$, with relations $q\preceq a'$ for $q$ in $Q_2$ and $a'\preceq q$ for $q$ in $\left(Q_1\right)_{\succ a}$ to form the poset $P^{(a)}$.
\end{enumerate}
\end{definition}
%%%%%%%%%%%%%%
\begin{center}
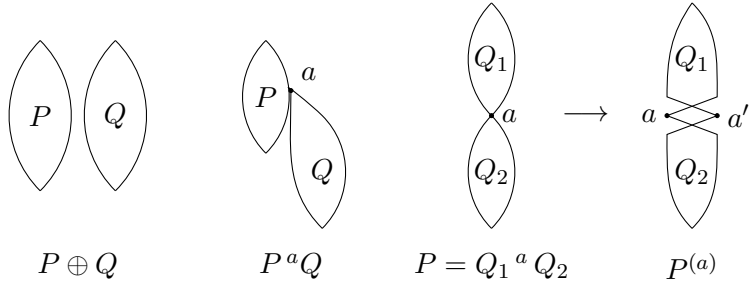
\begin{figure}[h]
\begin{tikzpicture}[scale=1]

\draw (0,0) to [out=45,in=-45] (0,2);
\draw (0,0) to [out=135,in=-135] (0,2);
\node at (0,1) {$P$};
\draw (1,0) to [out=45,in=-45] (1,2);
\draw (1,0) to [out=135,in=-135] (1,2);
\node at (1,1) {$Q$};
\node at (1/2,-1) {$P\oplus Q$};
%%%%%%%%%%%%%%
\draw (3,1/2) to [out=45,in=-45] (3,2);
\draw (3,1/2) to [out=135,in=-135] (3,2);
\node at (3,5/4) {$P$};
\filldraw (3+1/3,4/3) circle (.75pt);
\node at  (3+1/3,4/3) [above right] {$a$};
\draw (4-1/4,0-1/2) to [out=45,in=-45] (4-1/4,1.5-1/2) to [out=135] (3+1/3,4/3);
\node at (4-1/4,3/4-1/2) {$Q$};
\draw (4-1/4,0-1/2) to [out=135,in=-90] (3+1/3,4/3);
\node at (3+1/3,-1) {$P\hang Q$};
%%%%%%%%%%%%%%
\draw (6,1) to [out=45,in=-45] (6,5/2);
\draw (6,1) to [out=135,in=-135] (6,5/2);
\node at (6,7/4) {$Q_1$};
\draw (6,-1/2) to [out=45,in=-45] (6,1);
\draw (6,-1/2) to [out=135,in=-135] (6,1);
\node at (6,1/4) {$Q_2$};
\filldraw (6,1) circle (.75pt);
\node at  (6,1) [right] {$a$};
\node at (6,-1) {$P=Q_1\hang \,Q_2$};

\node at (7+1/4,1) {$\longrightarrow$};
\node at (8+2/3,-1) {$P^{(a)}$};
\node at (8+2/3,7/4) {$Q_1$};
\node at (8+2/3,1/4) {$Q_2$};

\draw (8+2/3,5/2) to [out=-45,in=90] (9,1+1/4) to (8+1/3,1);
\draw (8+2/3,5/2) to [out=235,in=90] (8+1/3,1+1/4) to (9,1);
\draw (8+2/3,-1/2) to [out=45,in=-90] (9,3/4) to (8+1/3,1);
\draw (8+2/3,-1/2) to [out=135,in=-90] (8+1/3,3/4) to (9,1);

\filldraw (9,1) circle (.75pt);
\node at  (9,1) [right] {$a'$};
\filldraw (8+1/3,1) circle (.75pt);
\node at  (8+1/3,1) [left] {$a$};

\end{tikzpicture}
\caption{The Operations}
\end{figure}
\end{center}

\begin{definition}
We define the graph $\Gamma(P)$ associated to a poset $P$ as having vertex set given by the connected order ideals $J_i$ of $P$ and edge set $E(\Gamma(P))$ given by the pairs of connected order ideals that intersect non-trivially. We write $\Gamma$ when the poset $P$ is clear.
\end{definition}
The edge ideal of the graph $\Gamma$ on $m$ vertices is the ideal \[I_\Gamma=\langle \,U_iU_j\,:\, \{i,j\}\in E(\Gamma)\,\rangle\;\subset \;K[U_1,\dots, U_m].\]

There exist term orders (\cite{PpartitionsRev} Theorem 1.4) such that the monomial ideal of $S$ generated by initial terms of generators of $I_P$ is precisely the edge ideal $I_\Gamma$ of the graph $\Gamma$ (under the obvious identification of variables $U_{J_i}$ and $U_i$).

\begin{proposition}\label{forestTwigs} A poset $P$ is a forest with duplication if and only if $\Gamma$ is a disjoint union of isolated edges and vertices.
\begin{proof} 
To show the forward direction, we prove that the three operations preserve the graph property. Note that the graph $\Gamma_*$ for the one element poset is a single vertex with no edges. 

Clearly, the vertex set of $\Gamma_{P\oplus Q}$ is the disjoint union of the vertex sets of $\Gamma_P$ and $\Gamma_Q$. If two ideals intersect non-trivially in $P\oplus Q$, then both ideals are contained in either $P$ or $Q$, so that the corresponding edge is internal to $\Gamma_P$ or $\Gamma_Q$. It follows that $\Gamma_{P\oplus Q}=\Gamma_P\oplus\Gamma_Q$.

Similarly, the vertex set of $\Gamma_{P\hang Q}$ is the disjoint union (under a slight abuse of notation, where for $J$ containing $a$ we identify the ideal $J\subseteq P$ with $J\cup Q$) of the vertex sets of $\Gamma_P$ and $\Gamma_Q$, since any connected order ideal of $P\hang Q$ intersecting $Q$ either is contained in or contains $Q$. Again edges of $\Gamma_{P\hang Q}$ are internal to $\Gamma_P$ and $\Gamma_Q$, so that $\Gamma_{P\hang Q}=\Gamma_P\oplus \Gamma_Q$.

The vertex set of $\Gamma_{P^{(a)}}$ is the vertex set of $\Gamma_P$ together with a vertex associated to the principal ideal $(a')\subset P^{(a)}$. The only ideal intersecting $(a')$ non-trivially is $(a)$, so that the edge set of $\Gamma_{P^{(a)}}$ is the edge set of $\Gamma_{P}$ together with the edge with endpoints $(a)$ and $(a')$. By construction, no ideal of $P$ intersects non-trivially with $(a)$, and so the degree of the vertex $(a)$ is exactly one.

To prove the other direction, we assume that $\widetilde{P}$ has the graph property and is not the result of any of the three operations: disjoint union, hanging, or duplication. We then prove that $\widetilde{P}$ is the one element poset. 

Let $x$ be a minimal element of $\widetilde{P}$. If there exists $p$ covering $x$, then there exists $q\neq p$ covering $x$, since $x$ is not hanging. Let $(p)$ denote the lower order ideal generated by $p$, and $[p]$ denote the upper order ideal, i.e., $[p]=\{y\in\widetilde{P}\,:\, y\succeq p\}$. If $r\in[p]\backslash[q]$, then the ideal $(q)$ has degree at least two in $\Gamma_{\widetilde{P}}$, since it meets $(p)$ and $(r)$. Thus, by symmetry, $[p]\backslash \{p\}$=$[q]\backslash \{q\}$. 

Let $r\in(p)\backslash(q)$ and $s\in(r)\cap(q)$. Then again $(q)$ has degree at least two, since it meets $(r)$ and $(p)$. It follows that $(r)\cap(q)$ is empty. Since $(r)$ does not hang below $p$, there exists $s\in\widetilde{P}\backslash[p]$ and $t\in(r)\cap(s)$. But in this case, $(p)$ has degree at least two since it meets $(s)$ and $(q)$. Unraveling the hypotheses, we see that $(p)\backslash \{p\}=(q)\backslash \{q\}$. 

Since $(p)\backslash \{p\}=(q)\backslash \{q\}$ and $[p]\backslash \{p\}$=$[q]\backslash \{q\}$ together imply that $p$ and $q$ form a duplication pair, we see that no elements cover $x$. Since $\widetilde{P}$ is connected (it is not a disjoint union), it follows that $\widetilde{P}$ consists of the single element $x$. Thus if a poset $P$ has $\Gamma_{P}$ a disjoint union of edges and vertices, and $P$ is not the result of disjoint union, hanging, or duplication, then $P$ is the singleton poset.
\end{proof}
\end{proposition}

\begin{theorem}\label{CIProposition}A finite poset $P$ is a forest with duplications if and only if the ring $S/I_P$ is a complete intersection. \end{theorem}
\begin{proof} From the preceding proposition, it is enough to show that $S/I_P$ is a complete intersection if and only if $\Gamma$ is a disjoint union of isolated edges and vertices.

The ideal $I_\Gamma$ is an initial ideal of $I_P$, with the consequence that $S/I_\Gamma$ and $S/I_P$ have the same Krull dimension. Since $I_\Gamma$ and $I_P$ have the same minimum number of generators, $S/I_P$ is a complete intersection if and only if $S/I_\Gamma$ is.

Let $\Gamma$ have vertex $J_1$ adjacent to vertices $J_2$ and $J_3$. Then the ring $S/I_\Gamma$ is not a complete intersection, since $U_{J_1}U_{J_3}$ is a zero divisor in $S/\left(U_{J_1}U_{J_2}\right)$. If $\Gamma$ has no vertices of degree greater than one, then $S/I_\Gamma$ can be written as a tensor product of univariate polynomial rings $K[U_J]$ corresponding to isolated vertices and quotient rings $K[U_{J_i},U_{J_j}]/(U_{J_i}U_{J_j})$ corresponding to edges $\{J_i,J_j\}$. The co-dimension of this ring is precisely the number of generators of $I_\Gamma$, hence $S/I_\Gamma$ is a complete intersection.
\end{proof}

\bibliographystyle{amsplain}
\bibliography{BrianBib}
\end{document}